\documentclass[leqno,12pt]{article}

\usepackage{amscd}
\usepackage{graphicx}
\usepackage[T1]{fontenc}
\usepackage{amsmath,amsfonts,amssymb,enumerate}
\usepackage[all,cmtip]{xy}
\usepackage[a4paper,top=2.5 cm,bottom=2 cm,left=2.8 cm,right=2.8 cm]{geometry}
\usepackage{amsthm} 
\usepackage[english]{babel}
\usepackage[latin1]{inputenc}
\usepackage{graphicx}
\newtheorem{theorem}{Theorem}[section]
\newtheorem{lemma}[theorem]{Lemma}
\newtheorem{prop}[theorem]{Proposition}
\newtheorem{corollary}[theorem]{Corollary}

\theoremstyle{definition}

\newtheorem{definition}[theorem]{Definition}
\newtheorem{remark}[theorem]{Remark}
\newtheorem{example}[theorem]{Example}

\newtheorem{conjecture}[theorem]{Conjecture}

\numberwithin{equation}{section}

\newcommand{\N}{\mathbb{N}}

\newcommand{\B}{\mathcal{B}}

\newcommand{\aap}{\mathrm {Ap}}
\newcommand{\ord}{\mathrm {ord}}
\newcommand{\gr}{\mathrm{gr}}
\newcommand{\m}{\mathfrak{m}}
\newcommand{\Hom}{\mathrm {Hom}}

\newcommand{\Ann}{\mathrm {Ann}}

\begin{document}




\title{Lefschetz Properties of Gorenstein Graded Algebras associated to the Ap\'ery Set of a Numerical Semigroup}

\author{Lorenzo Guerrieri 
\thanks{Universit\`a di Catania, Dipartimento di Matematica e
Informatica, Viale A. Doria, 6, 95125 Catania, Italy.   \textbf{Email address:} guelor@guelan.com  }}

\maketitle


\begin{abstract}
\noindent In this paper we study the Weak Lefschetz property of  two classes of standard graded Artinian Gorenstein algebras associated in a natural way to the Ap\'ery set of numerical semigroups. To this aim we also prove a general result about the transfer of the Weak Lefschetz property from an Artinian Gorenstein algebra to its quotients modulo a colon ideal.
\medskip

\end{abstract}


\section*{Introduction}

The Lefschetz properties for standard graded Artinian $K$-algebras are algebraic concepts introduced by Stanley in \cite{stanley}, motivated by the Hard Lefschetz Theorem on the cohomology rings of smooth irreducible complex projective varieties. 
The notion of Poincar\'{e} duality for these cohomology rings inspired the definition of Poincar\'{e} duality for algebras which is equivalent to the Gorensteiness. Hence many results about Lefschetz properties have been proved in the Gorenstein case. \\
 In \cite{dilw}, it has been shown that almost all Artinian Gorenstein algebras have the Strong Lefschetz property. But in general it is a difficult problem to know whether a given specific algebra has the Strong (or the Weak) Lefschetz property. \\
Using Macaulay-Matlis duality in characteristic zero it is possible to present Artinian Gorenstein algebras in the form $A = \frac{Q}{\Ann_Q(F)}$ with $F \in R = K[x_1,\ldots, x_n]$ a homogeneous polynomial and $Q = K[X_1,\ldots, X_N]$ where $X_i:= \frac{\partial}{\partial x_i}$ are differential operators (e.g. \cite{maewat}).
In \cite{wa2} and \cite{maewat}, using this presentation of the algebras, the autors introduced a criterion based on determinants of Higher Hessians that establishes whether an algebra has or not Lefschetz properties. \\
In \cite{rodrigo} this criterion was used to construct explicit examples of Artinian Gorenstein algebras that do not satisfy one or both Lefschetz properties. \\
Even if the "Lefschetz properties problem" has a very simple formulation, it is in general open even in low codimension. Indeed, while in codimension two it is known that all the Artinian Gorenstein graded algebras have the Strong Lefschetz property, in codimension three this is not known but is conjectured to be true. The first examples of algebras without Strong or Weak Lefschetz properties appear in codimension four. \\
Furthermore, we do not know if there are examples of algebras without one or both the Lefschetz properties (in any given codimension) that belong to the smaller class of Complete Intersection rings. Indeed it is conjectured that all the Complete Intersection Artinian graded algebras have the SLP. For all these results and open conjectures we refer to the monography \textit{The Lefschetz properties}  \cite{monografia}.

In this work we study the Weak Lefschetz property (WLP) for a class of graded Artinian Gorenstein algebras built up starting from the Ap\'ery Set of a numerical semigroup and of which they reflect the lattice structure. Our goal is to study whether these algebras have the WLP in codimension three and in the Complete Intersection case. In both cases we are going to give a criterion implying the WLP under some extra assumption.
The structure of this paper is the following: in Section 1 we recall some definitions and known results about Lefschetz properties. \\ 
In Section 2 we prove a key theorem (\ref{main1}) stating that, if a graded Artinian Gorenstein algebra $$A \cong \dfrac{K[x_1,\ldots, x_n]}{I}$$ has the WLP and some conditions on the socle degree or on the Hilbert function of $G$ are fulfilled, then also the quotient ring $$ \dfrac{A}{(0:_A L)} $$ is Gorenstein and it has the WLP for any linear element $L \in A$. \\ 
In general it is known that if an Artinian graded algebra $A$, not necessarily Gorenstein, has the SLP and if $L \in A$ is a Lefschetz element for $A$ (see definition in Section 1), then $ \dfrac{A}{(0:_A L)} $ has the SLP; moreover, if $A$ has the SLP in the narrow sense (see Definition \ref{deflef}), then for any linear element $L_1 \in A$, the quotient ring $ \dfrac{A}{(0:_A L_1)} $ has the SLP if it has the same Hilbert function as $ \dfrac{A}{(0:_A L)} $ \cite[3.11 and 3.40]{monografia}.
We have looked here for a similar result about WLP under the hypothesis of Gorensteiness of $A$.  

In Section 3 we construct the graded Artinian algebra associated to the Ap\'ery Set of a numerical semigroup. This is the same ring that appeared in \cite{bryant} and in \cite{da-mi-sa} and that was used to prove results about the Gorensteiness and the Complete Intersection property of the associated graded ring of a semigroup ring. We study here how to present these algebras as quotients of the algebra of differential operators by the annihilator of a homogeneous polynomial, following the known result by Maeno and Watanabe \cite[Theorem 2.1]{maewat}. At the end of this section we discuss how our Theorem \ref{main1} can be applied to this class of rings when they satisfy a \it Quotient Condition\rm, stating that they are expressible as quotients of an algebra $G$ satisfying the WLP by the ideal $ (0:_A x^C) $,  for a linear monomial $x \in G$ and an integer $C \geq 0$. \\
In Section 4 and 5 we show that the algebras associated to Ap\'ery Sets of numerical semigroups in codimension three and in the Complete Intersection case satisfy this Quotient Condition. In Section 4 we deal with the Complete Intersection case and we recall results from \cite{da-mi-sa} to establish when our algebras associated to 
Ap\'ery Sets are Complete Intersections. Then we use an useful known criterion about WLP for Complete Intersection rings to get our result. \\
Finally, in Section 5 we assume the codimension to be three; in this case we are able to completely characterize the defining ideal of all the graded algebras associated to Ap\'ery Sets in function of their socle degree and we find that any such algebra is of the form $$ A= \dfrac{K[y, z, w]}{I} = \dfrac{G}{(0:_G z^C)} $$ with $G$ a Complete Intersection Artinian graded ring and $C$ a positive integer. 

\bigskip

\section{Lefschetz Properties}

We start recalling definitions and important results about the Lefschetz properties. 

Let $K$ be a field of characteristic zero and let $A= \bigoplus_{i = 0}^D A_i$ be a standard graded Artinian $K$-algebra. Since it is Artinian, $A$ is a finite dimensional $K$-vector space.  

Consider the polynomial ring in $n$ variables $K[x_1,\ldots, x_n]$ with $n \geq 1$. We can always write $$ A \cong \dfrac{K[x_1,\ldots, x_n]}{I} $$ where $I \subseteq K[x_1,\ldots, x_n]$ is an homogeneous ideal of height $n$. Since $A$ is Artinian, the integer $n=\dim_K(A_1)$ is the \textit{codimension} of the ring $A$.

\bigskip


\begin{definition} \label{deflef}
We say that: 
\begin{itemize}
 \item{$A$ has the Weak Lefschetz property (WLP) if there is an element $L \in A_1$ such that the multiplication map $\times L : A_i \to A_{i+1} $ has maximal rank for every $i=0, \ldots, D-1$.}
 \item{$A$ has the Strong Lefschetz property (SLP) if there is an element $L \in A_1$ such that the multiplication map $\times L^d : A_i \to A_{i+d} $ has maximal rank for every $i=0, \ldots, D$ and $d=0, \ldots, D-i$.} 
 \item{$A$ has the Strong Lefschetz property in the narrow sense if there is an element $L \in A_1$ such that the multiplication map $\times L^{D-2i} : A_i \to A_{D-i} $ is bijective for every $i=0, \ldots, \lfloor \frac{D}{2} \rfloor$.}
\end{itemize}
\end{definition}

A linear form $L \in A_1$ such that each map $\times L : A_i \to A_{i+1} $ has maximal rank is called a \textit{Weak Lefschetz element}. If instead each map $\times L^d : A_i \to A_{i+d} $ has maximal rank, $L$ is called a \textit{Strong Lefschetz element}. 

\bigskip

The ring $A=\bigoplus_{i = 0}^D A_i$ is Gorenstein if there is a perfect pairing of its homogeneous components, that is $A_i \cong A_{D-i}$ for every $i$.

Hence, if $A$ is Gorenstein, it has a symmetric Hilbert function, that means: $$  \dim_K(A_i)= \dim_K(A_{D-i}), \hspace{0.5cm} \forall i. $$
We call the integer $D$ the \textit{socle degree} of $A$.

In this work we are always dealing with Gorenstein algebras and in this case it is known that the SLP is equivalent to the Strong Lefschetz property in the narrow sense \cite{maewat}. 

The Artinian ring $A\cong \dfrac{K[x_1,\ldots, x_n]}{I}$ is a Complete Intersection (CI) if $I$ is minimally generated by exactly $n$ elements (notice that in general $I$ is generated by at least $n$ elements). It is well known that a Complete Intersection ring is always Gorenstein.
Here we list some known results and still open problems about Lefschetz properties of Gorenstein rings. \\
All the details about these topics can be found in \cite[3.15, 3.48, 3.35, 3.46, 3.80]{monografia}. \\

Let $A\cong \dfrac{K[x_1,\ldots, x_n]}{I}$ be a standard graded Artinian Gorenstein $K$-algebra, then:

\begin{itemize}
 \item{If $n=2$, then $A$ has the SLP.}
 \item{If $n=3$ and $A$ is a CI, then $A$ has the WLP.}
 \item{For every $n$, if $A$ is a CI and $I$ is a monomial ideal, then $A$ has the SLP.} 
 \item{It is conjectured that if $A$ is a CI, then $A$ has the SLP.}
 \item{If $n=3$, it is unknown whether there exists a such ring $A$ that does not have the WLP or the SLP.} 
 \end{itemize}

Important tools needed to study whether a Gorenstein algebra has the Lefschetz properties are the higher Hessians. \\
We give now some definitions and results taken from a paper by Maeno and Watanabe \cite{maewat} and from a more recent paper by Gondim and Zappal\'{a} \cite{rodzap}.

\begin{prop}
\label{weaky}
Let $A$ be a standard graded Artinian Gorenstein $K$-algebra and $k:= \lfloor \frac{D}{2} \rfloor $ where $D$ is the socle degree of $A$. Thus we have: 
\begin{enumerate}
\item
If $D$ is an odd number, $A$ has the WLP if there is an element $L \in A_1$ such that the multiplication map $\times L : A_k \to A_{k+1} $ is an isomorphism.
\item
If $D$ is an even number, $A$ has the WLP if there is an element $L \in A_1$ such that the multiplication map $\times L : A_k \to A_{k+1} $ is surjective or equivalently the multiplication map $\times L : A_{k-1} \to A_{k} $ is injective.
\end{enumerate}
\end{prop}

We define the ring of differential operators $Q:=K[X_1,\ldots,X_n]$ where $$X_i:=\dfrac{\partial}{\partial x_i},$$ and we state a well known result about the representation of a standard graded Artinian Gorenstein $K$-algebra as a quotient of the ring $Q$ modulo the annihilator of a polynomial in $K[x_1,\ldots, x_n]$. This result follows in general from the theory of inverse system \cite{BH}, but we refer explicitly to \cite{maewat} where is given a more direct proof of it.

\begin{theorem} \textup{\cite[Theorem 2.1]{maewat}} 
\label{annq}
Let $A$ be a standard graded Artinian Gorenstein $K$-algebra. Then there exists a polynomial $F \in K[x_1,\ldots, x_n]$ such that $A$ is isomorphic to the quotient $$ \frac{Q}{\Ann_Q(F)}. $$ 
\end{theorem}

This classical result shows that $A$ is generated over $K$ exactly by the monomials in $K[x_1,\ldots, x_n]$ that do not annihilate $F$ when considered as differential operators.

\begin{definition}
\label{hes}
Let $F$ be a polynomial in $K[x_1,\ldots, x_n]$ and $d \geq 1$ an integer. Take a $K$-linear basis $\B_d= \lbrace \alpha_i  \rbrace_{i = 1}^s$ of $A_d$. \\
We define the $d$-th Hessian of $F$ as the matrix $$  Hess^{d}_{\B_d}(F):= \{ ( \alpha_i(X)\alpha_j(X)F(x))_{i,j=1}^{s}\}. $$ We call $ hess^{d}_{\B_d}(F) $ the determinant of this matrix.  The singularity of the matrix is independent of the chosen basis and hence we can write simply $ Hess^{d}(F) $ and $ hess^{d}(F) $. Clearly $hess^{d}(F)$ is a polynomial in $K[x_1,\ldots, x_n]$.
\end{definition}


\begin{definition}
\label{mixhes}
Let $F$ be a polynomial in $K[x_1,\ldots, x_n]$. \\
Taking two integers $d,t \geq 1$ and two bases of $A_d$ and $A_t$, we define the mixed Hessians of the polynomial $F$ as $$ Hess^{d,t}(F):= \{ ( \alpha_i(X)\beta_j(X)F(x))\} $$ where $ \{ \alpha_i \}_{i=1}^{s_1} $ and $ \{ \beta_j \}_{j=1}^{s_2} $ form respectively the basis of $A_d$ and $A_t$.
\end{definition}


\begin{theorem} \textup{\cite[Theorem 2.10]{rodzap}}
\label{hess}
Let $A$ be a standard graded Artinian Gorenstein $K$-algebra and $k:=\lfloor \frac{D}{2} \rfloor$ where $D$ is the socle degree of $A$. Thus we have: 
\begin{enumerate}
\item
The algebra $A=\frac{Q}{\Ann_Q(F)}$ has the SLP if and only if all the Hessians $Hess^{d}(F)$ for $d=1, \ldots, k$, have maximal rank (hence if they have nonzero determinant). 
Moreover, a linear form $L= \sum a_i x_i \in A_1$ is a Strong Lefschetz element if $F(a_1, \ldots, a_n) \neq 0$ and $hess^{d}(F)(a_1, \ldots, a_n)$ is nonzero for all $d$.
\item
If $D$ is an odd number, the algebra $A=\frac{Q}{\Ann_Q(F)}$ has the WLP if and only if the maximal Hessian $Hess^{k}(F)$ has nonzero determinant. Moreover, a linear form $L= \sum a_i x_i \in A_1$ is a Weak Lefschetz element if $F(a_1, \ldots, a_n) \neq 0$ and $hess^{k}(F)(a_1, \ldots, a_n)$ is nonzero.
\item
If $D$ is an even number, the algebra $A=\frac{Q}{\Ann_Q(F)}$ has the WLP if and only if the mixed Hessian $Hess^{k-1, k}(F)$ has maximal rank. Moreover, a linear form $L= \sum a_i x_i \in A_1$ is a Weak Lefschetz element if $F(a_1, \ldots, a_n) \neq 0$ and the matrix $Hess^{k-1, k}(F)(a_1, \ldots, a_n)$ has maximal rank.
\end{enumerate}
\end{theorem}




\bigskip

\section{WLP of quotient algebras}

Let $G= \bigoplus_{i = 0}^D G_i$ be a standard graded Gorenstein Artinian $K$-algebra that has the WLP. \\
We prove a theorem that, under some assumption on the Hilbert function of $G$, will allow us to transfer the Weak Lefschetz Property to some of the quotients of $G$. Also in this section $K$ is a field of characteristic zero.




\begin{theorem}
\label{main1} 
Let $G= \bigoplus_{i = 0}^D G_i \cong \dfrac{K[x_1,\ldots, x_n]}{I} $ be a standard graded Gorenstein Artinian $K$-algebra that satisfies the WLP. \\
Then, for $l=1, \ldots, n$, the quotient ring $$ A = \dfrac{G}{(0:_G x_l)} $$ is also a standard graded Gorenstein Artinian $K$-algebra. Assume that $A$ and $G$ have the same codimension and let $ k:=\lfloor \frac{D}{2} \rfloor $. Then:
\begin{enumerate}
\item If the socle degree $D$ of $G$ is odd, then also $A$ has the WLP.
\item If the socle degree $D$ of $G$ is even and $\dim_K(G_{k-1})=\dim_K(G_k)$, then also $A$ has the WLP.
\end{enumerate}
\end{theorem}

\begin{proof}
Since $G$ is Gorenstein, there is a perfect pairing of its homogeneous components, that is $G_d \cong G_{D-d}$ for every $d$. These isomorphisms are obtained associating to a homogeneous element $f \in G_d$, the element $\varphi \in G_{D-d}$ such that $ f \varphi = q $ where $ q $ is the generator of the socle $G_D$.
We observe that the socle of $A$ is the homogeneous component $A_{D-1}$ and its unique generator (modulo $I$) as $K$-vector space is the monomial $x_l^{-1}q \in G_{D-1}$. Therefore it is easy to see that $A$ is a standard graded Gorenstein Artinian algebra.

Since $G$ and $A$ have the same codimension, let $L = \sum_{j=1}^{n} a_j x_j \in G_1= A_1$ be a Weak Lefschetz Element for $ G $ 

We are going to use the characterization of WLP given in Proposition \ref{weaky}. 

\textbf{(1) $D$ odd:} \\
By Proposition \ref{weaky}, we have that the multiplication map $\times L : G_k \to G_{k+1} $ is an isomorphism and we want to prove that the map $\times L : A_k \to A_{k+1} $ is surjective.

The ideal $J=(0:_G x_l)$ can be seen as a  $K$-vector subspace of $G$ and by definition $ A \cap J = (0)$. Hence $ G \cong A \oplus J $ as $K$-vector spaces and thus we can write the elements of $G$ in the form $(a,j)$ with $a \in A$ and $j \in J$ and we have $(a,j) \in J$ if and only if $a=0$. 

Take $(a,j) \in  G_{k+1}$ with $a \neq 0$. The map $\times L$ is an isomorphism, so we can consider its preimage $ \times L^{-1}(a,j)= (a_1, j_1) \in G_k$. Showing $a_1 \neq 0$, we will have that $\times L : A_k \to A_{k+1} $ is surjective. \\
Assume $a_1 =0$, then by definition $(a,j)= \times L(a_1,j_1)= \times L(0,j_1)= (0,Lj_1) \in J$, and thus $a=0$. This is a contradiction. \\
\textbf{(2) $D$ even:} \\
We proceed similarly as in the preceding case. By the assumption on the dimensions of $ G_k $ and $ G_{k-1} $, we have that the multiplication map $\times L : G_{k-1} \to G_{k} $ is an isomorphism and, by Proposition \ref{weaky}, we need to prove that the map $\times L : A_{k-1} \to A_{k} $ is surjective (or injective because $A_{k-1} \cong A_{k}$). We can therefore repeat the proof of the preceding case. 
\end{proof}

By using a linear change of coordinates on $ x_1,\ldots , x_n $, we find as an easy corollary that, if $G$ is an Artinian standard graded Gorenstein algebra with the WLP satisfying the assumption of Theorem \ref{main1}, then any quotient $ \dfrac{G}{(0:_G f)} $ with $f \in G_1$ has the WLP. 

\begin{corollary}
\label{maincor} 
Let $G= \bigoplus_{i = 0}^D G_i \cong \dfrac{K[x_1,\ldots, x_n]}{I} $ be a standard graded Gorenstein Artinian $K$-algebra with the WLP and let $f \in G_1$ a linear element. Take the same assumptions of Theorem \ref{main1}. \\
Then, the quotient ring $$ A = \dfrac{G}{(0:_G f)} $$ is also a standard graded Gorenstein Artinian $K$-algebra. If $A$ and $G$ have the same codimension, then also $A$ has the WLP.
\end{corollary}

\begin{proof}
Write $f= \sum b_i x_i$ and, assuming $b_1 \neq 0$, make the linear change of coordinates $ \varphi: K[x_1,\ldots, x_n] \longrightarrow K[y_1,\ldots, y_n] $ given by $y_1:=\varphi(f)$ and $y_i:=\varphi(x_i)$ for $i \geq 2$. Consider the surjective homomorphism $$  K[x_1,\ldots, x_n] \stackrel{\varphi} \longrightarrow K[y_1,\ldots, y_n] \twoheadrightarrow \dfrac{K[y_1,\ldots, y_n]}{\varphi(I)}=: G^{\prime}$$ whose kernel is the ideal $I$. \\
Therefore $G \cong G^{\prime}$. Moreover the ideal $ (0:_{G^{\prime}} y_1) $ of $G^{\prime}$ is the image of the ideal $(0:_G f)$ of $G$. Hence the result follows applying Theorem \ref{main1} to the ring $ \dfrac{G^{\prime}}{(0:_{G^{\prime}} y_1)} $.
\end{proof}

Next example shows that the result of Theorem \ref{main1} is not true in general.

\begin{example}
\label{wachi}
Let $R = K[a, b, x, y, z]$ be the polynomial ring in $5$ variables over a field $K$ (of characteristic zero)and let $Q$ be the corresponding ring of differential operators. We use the presentation of standard graded Artinian Gorenstein $K$-algebras given in Theorem \ref{annq} in order to define two of such rings. Let $f = a^2x + aby + b^2z$, $g = (a^2x + aby + \frac{b^2z}{2})z$ and $I = \Ann_Q(f)$, $J = \Ann_Q(g)$. We consider the $K$-algebras $A= R/I$ and $G = R/J$.

In \cite[p. 138, Example 3.78]{monografia} is shown that $hess^1(f)=0$ and none of
the five variables in $A$ can be eliminated by a linear transformation. Hence $A$ does not have both the SLP and the WLP by Theorem \ref{hess} (notice that the socle degree of $A$ is $3$).

Now, since $ \frac{\partial}{\partial z}g = f $, we have that $$ I = \Ann_Q(f) = \Ann_Q \left( \frac{\partial}{\partial z}g \right)
= \Ann_Q(g) :_R z = J :_R z. $$ Thus, $$A = \dfrac{G}{(0:_G z)}.$$

The $K$-algebra $G$ has socle degree $4$ and Hilbert vector $(1, 5, 10, 5, 1)$. Computing the first Hessian $Hess^1(g)$, it is possible to show that $G$ has the SLP and therefore the WLP.
 \end{example}

\bigskip

\section{Algebras associated to Ap\'ery Sets}

We want to investigate the Lefschetz properties of a class of Artinian Gorenstein algebras obtained from numerical semigroups. 

Let $S= \langle g_1=m, g_2, \ldots, g_n \rangle \subseteq \N$ be a numerical semigroup. Recall that in this case gcd$(g_1, \ldots, g_n) = 1$. For all the basic knowledge about numerical semigroups and semigroup rings consider as references \cite{ros-san} and \cite{fro}. 

The Ap\'ery set of $S$ with respect to the minimal generator of the semigroup is defined as the set $$  \aap(S):= \{  s \in S: s-g_1 \not \in S \} = \{0=\omega_1 < \omega_2 < \dots < \omega_m= f+ g_1\}, $$ where $f:=$max$(\mathbb N \setminus S)$ is the Frobenius number of $S$. Note that $\aap(S)$ is a finite set and $  |\aap(S)|=g_1=m $.

\begin{definition}
Let $s \in S$. A representation of $s$ is an $n$-uple $ \lambda=(\lambda_1, \ldots, \lambda_n)$ such that $ s=\sum_{i = 1}^n \lambda_i g_i$. The order of $s$ is defined as $$  \ord(s):= \mbox{max}  \{  \sum_{i = 1}^n \lambda_i: \lambda \mbox{ is a representation of } s \}.$$
A representation is said to be maximal if $\ord(s)=\sum_{i = 1}^n \lambda_i.$
\end{definition}

\begin{definition}
\label{mpurosimm}
The semigroup $S$ is said $M$-pure symmetric if for each $i=0, \ldots, m$: \\
(1) $ \omega_{i}+\omega_{m-i}=\omega_{m} $ and \\
(2) $\ord (\omega_{i}) + \ord (\omega_{m-i}) = \ord(\omega_{m})$.
\end{definition}

Therefore the Ap\'ery set of a $M$-pure symmetric semigroup has the structure of a symmetric lattice.

Let $K$ be a field of characteristic zero and consider the homomorphism: $$ \begin{array}{ccccc}

 \Phi: K[x_1,\ldots, x_n] \longrightarrow K[t] \\
& & & &  \\
 x_i  \longmapsto t^{g_i}. \\
\end{array}
$$ The ring $ R=K[S]:=K[t^{g_{1}},\dots,t^{g_{n}}]\cong \dfrac{K[x_1,\ldots, x_n]}{ker(\Phi)}$ is a one dimensional ring associated to the semigroup $S$.  \\
We can associate to any representation of an element $ s \in S $ a monomial in $R$ by the correspondence $$  s= \sum_{i = 1}^n \lambda_i g_i \longleftrightarrow x^{\lambda}:= x_1^{\lambda_1} \cdots x_n^{\lambda_n}. $$ induced by the preceding homomorphism. We can observe that the monomials in $R$ that correspond to different representations of the same element $s$ are equivalent modulo $ ker(\Phi) $.

Consider now the zero dimensional ring $ \overline{R}:=R/x_1R$. \\
Notice that for $s=\sum_{i = 1}^n \lambda_i g_i \in \aap(S)$, we have $\lambda_1=0$ and the corresponding monomial $ \prod_{i=2}^n \overline{ x_i}^{\lambda_i} \neq 0 $ in $\overline{R}$. Conversely if $s \not \in \aap(S)$, then $x_1$ divides $x^{\lambda}$ for at least one representation 
$ \lambda $ of $s$ and hence $$ \overline{R}= \langle \overline{x^{\lambda}}  \ | \sum_{i = 1}^n \lambda_i g_i \in \aap(S) \rangle_K $$ is generated as a $K$-vector space by the classes modulo the ideal $x_1R$ of the monomials $ x^{\lambda} $ for every representation $ \lambda $ of any element of $ \aap(S) $. Notice that in this way there is a one to one correspondence between the elements of $\aap(S)$ and the generators of $ \overline{R} $ as a $K$-vector space.

Recall that for a graded ring $R$ and a homogeneous ideal $ \m $, the associated graded ring of $R$ with respect to $\m$ is defined as $$ \gr_{\mathfrak m}(R):= \bigoplus_{i \geq 0} \dfrac{\m^i}{\m^{i+1}}. $$

\begin{definition}
\label{assalg}
Let $ \overline{\mathfrak m} $ be the maximal homogeneous ideal of $\overline{R}$. Define $$ A= \gr_{\overline{\mathfrak m}} (\overline{R}) $$ to be the \textit{associated graded algebra of the Ap\'ery set of $S$}. 
\end{definition}

By definition of associated graded ring, we have that $$  A= \bigoplus_{i = 0}^D A_i = \langle \overline{x^{\lambda}} \ | \sum_{i = 1}^n \lambda_i g_i \in \aap(S) \mbox{ and } \lambda \mbox{ is maximal }\rangle_K$$ is an Artinian standard graded $K$-algebra generated by the monomials $ \overline{x^{\lambda}}$ associated to the maximal representations of the elements of $ \aap(S) $. Notice that the socle degree is $D= \ord(f+g_1)$ and that we can think of the homogeneous $K$-generators of $A$ to have the same lattice structure as $\aap(S)$.
In the work 
\cite{bryant}, Lance Bryant characterized when this kind of rings are Gorenstein.

\begin{prop}
\label{bryant}
Let $S$ be a numerical semigroup. Then the ring $A$ associated to $\aap(S)$ is Gorenstein if and only if $S$ is $M$-pure symmetric.
\end{prop}

\begin{example}
\label{ex1}
Consider the numerical semigroup $S= \langle 8, 10, 11, 12 \rangle$. \\
Its Ap\'ery set is $\aap(S)= \lbrace 0, 10, 11, 12, 21, 22, 23, 33 \rbrace$ and it can be easily checked that $S$ is $M$-pure symmetric. The associated graded Artinian algebra is $$  A= K \oplus \langle y, z, w \rangle K  \oplus \langle yz, yw, zw \rangle K \oplus \langle yzw \rangle K. $$ Since in the semigroup $ 22= 11+11= 10+12 $, then in the ring $A$, $ yw \equiv z^2 $ and hence $$  A  \cong \dfrac{K[y, z, w]}{(y^2, z^2 - yw, w^2)}.$$
\end{example}

\bigskip

 In order to check when this kind of graded rings have the Lefschetz properties we want to use Theorem \ref{hess} and some consequences of it and therefore we need to find the polynomial $F$ such that $A$ is isomorphic to $\frac{Q}{\Ann_Q(F)}$ (see Theorem \ref{annq}) and compute its Hessians.
 
It is possible to write the Ap\'ery set of $S$ as $ \aap(S)= \bigcup_{d \geq 0}^D Ap_d $ where $Ap_d$ is the set of the element of $\aap(S)$ of order $d$. It is clear from the definition of $ A $ that the dimension of $A_d$ as $K$-vector space is equal to the cardinality of $Ap_d$. 

From now on we are going to use the previous notation $x^{\lambda}$ for the homogeneous monomials in $A$ instead of the heavier notation $ \overline{x^{\lambda}} $ (or we will use the variables $y, z, w$ if in codimension 3).
We can choose as basis $\B_d$ of $ A_d $ a set of of monomials $ \lbrace x^{\lambda^{1}}, \ldots, x^{\lambda^{b_d}} \rbrace $ where the $ \lambda^{j} $ are maximal representations of the elements $ \omega_{j} \in Ap_d$ and $ b_d= \dim_K(A_d) $. For instance in Example \ref{ex1} we can choose equivalently as basis for $A_2$ either the set $ \lbrace yz, yw, zw \rbrace $ or the set $ \lbrace yz, z^2, zw \rbrace $ since $ yw \equiv z^2 $ in such ring $A$.

\begin{theorem}
\label{Fnum}
Call $ \Lambda $ the set of the maximal representations of the maximal element of $\aap(S)$, $f+g_1$. \\
The graded ring $A$ associated $ \aap(S) $ is isomorphic to $\frac{Q}{\Ann_Q(F)}$, where $Q=K[X_1,\ldots,X_n]$ and $F=\sum_{\lambda \in \Lambda} x^{\lambda}$.
\end{theorem}

\begin{proof}
We want to apply to this particular case of graded rings associated to the Ap\'ery set of a semigroup the general proof of the existence of the polynomial $F$ given by Maeno and Watanabe \cite[Theorem 2.1]{maewat}. \\
Identifying the algebra $A$ with the quotient of $Q$ by an ideal $I$ (called the defining ideal of $A$), we have the exact sequence of modules $ Q \to A \to 0 $. That sequence induces another exact sequence $$ 0 \to \Hom(A,K)\cong A \stackrel{\theta} \to \Hom(Q,K)\cong K[[x_2, \ldots, x_n]]. $$
Maeno and Watanabe proved that $F$ is equal to $ \theta(1)\in K[[x_2, \ldots, x_n]] $ with $1:=(1,0, \ldots, 0) \in A$. \\
In order to use this fact we recall that the isomorphism between the ring $A$ and $ \Hom(A,K) $ is the application that maps $a=(a_0, \ldots, a_D) \in A$ to the map $ \varphi_{a}: A \to K $ defined by $ \varphi_{a}(c)= \sum_{i=0}^{D} a_{i} c_{D-i} $ for each $c=(c_0, \ldots, c_D) \in A$. 
We also recall that the isomorphism between $ K[[x_2, \cdots, x_n]] $ and $ \Hom(Q,K) $ is obtained identifying a formal power series $ f$ with the homomorphism which maps every monomial of $Q$ to its corresponding numerical coefficient in the power series $f$. \\
Thus we have $1 \in A$ identified with the homomorphism $ \varphi_{1} $ mapping $c=(c_0, \ldots, c_D) \in A$ to its last component $c_D$. Hence we compute $$ F=\theta(1)=\sum \overline{\varphi_1}(X_2^{s_2} \cdots X_n^{s_n}) x_2^{s_2} \cdots x_n^{s_n},$$ where the sum is taken over the infinite basis of $Q$ over $K$ and $ \overline{\varphi_1}(\alpha):= \varphi_1(\alpha+I) $ for $ \alpha \in Q $. \\
Now we compute $ \overline{\varphi_1}(X_2^{s_2} \cdots X_n^{s_n}) $ for all possible values of the $ s_i $. Let $ \alpha= X_2^{s_2} \cdots X_n^{s_n} $ and $ s= \sum_{i=1}^{n} s_i g_i \in S $. If $ s \not \in \aap(S) $ or $ s \in \aap(S) $ but $ \ord(s) > \sum_{i=1}^{n} s_i $, then $\alpha \in I$ and $ \overline{\varphi_1}(\alpha)=0 $. \\
All the other monomials of the $K$-basis of $Q$ are associated to a maximal representation of an element $ \omega \in Ap_d$, therefore for every $ \alpha= X_2^{s_2} \cdots X_n^{s_n} $ there exists a monomial $ X^{\lambda} $ such that $ \alpha \equiv X^{\lambda}$ mod $I$. \\ 
If $d<D$, the $D$-th component of $ X^{\lambda} $ is equal to zero and in this case $\overline{\varphi_1}(\alpha)=\overline{\varphi_1}(X^{\lambda})=0 $. If instead $d=D$, it is clear that $\alpha$ is a maximal representation of $ f+m $ and $ \overline{\varphi_1}(\alpha)=\overline{\varphi_1}(X^{\lambda})=1 $.  Thus in the sum only the maximal representations of $f+g_1$ with coefficient $1$ survive. 
\end{proof}

\begin{example}
\label{ex2}
Let $S= \langle 16, 18, 21, 27 \rangle$.
We compute the Hessians to study the Lefschetz properties of the algebra $A$ associated to the Ap\'ery set of $S$. \\
We have that $\aap(S)= \lbrace 0, 18, 21, 27, 36, 39, 42, 45, 54, 57, 60, 63, 72, 78, 81, 99 \rbrace$ and $S$ is $M$-pure symmetric.

Doing the computation as in Example \ref{ex1} we see that in this case $$  A  \cong \dfrac{K[y, z, w]}{(y^5, z^3 - y^2w, w^2, zw, y^3z)}.$$
Hence the socle degree of $ A $ is $D=5$. In the semigroup $99=4 \cdot 18 + 27= 2 \cdot 18 + 3\cdot 21$ and by \ref{Fnum} the polynomial $F=y^4w+y^2z^3$. 
We choose as bases respectively $ \lbrace y, z, w \rbrace $ for $ A_1 $ and $ \lbrace y^2, yz, z^2, yw  \rbrace $ for $ A_2 $.

We compute the first Hessian of $F$, $$  Hess^{1}(F)=\begin{pmatrix}
y^2w+ z^3 & yz^2 & y^3 \\ 
yz^2 & zy^2 & 0 \\
y^3 & 0 & 0
\end{pmatrix}  
 $$



The second Hessian is $$Hess^{2}(F)=\begin{pmatrix}
 w  & 0 &  z &  y \\ 
0 &  z &  y & 0 \\
 z &  y & 0 & 0 \\
 y & 0 & 0 & 0
\end{pmatrix}  
$$

The hessians have both maximal rank, hence $A$ has the SLP.

\end{example}

We give the following definition, which allows us to apply the result of Theorem \ref{main1} to the algebras associated to Ap\'ery sets.

\begin{definition}
\label{discesa}
Let $A$ be a standard graded Gorenstein Artinian $K$-algebra. We say that $A$ satisfies the \it Quotient Condition \rm, if there exists a standard graded Gorenstein Artinian $K$-algebra $G$ having the WLP such that $$ A = \dfrac{G}{(0:_G x_l^C)} $$ for some linear monomial $x_l$ and some integer $C \geq 0.$ (Notice that when $C=0$, $A=G$ has the WLP.)
\end{definition}

If a standard graded Gorenstein Artinian $K$-algebra $A$ satisfies the Quotient Condition, by induction $$  A = \dfrac{G}{(0 :_G x_l^C)}= \dfrac{\dfrac{G}{(0 :_G x_2^{C-1})}}{(0 :_{G^{\prime}} \overline{x_l})} $$ with $G^{\prime}:=\dfrac{G}{(0 :_G x_l^{C-1})}$. If the hypothesis of Theorem \ref{main1} are satisfied at each step of the induction, then $A$ has the WLP.

In the next sections, we are going to show that all the algebras associated to Ap\'ery sets which are either Complete Intersection or of codimension 3 satisfy the Quotient Condition. Hence, whenever one of these algebras $A = \dfrac{G}{(0:_G x_l)}$ is a quotient of an algebra $G$ satisfying the hypothesis of Theorem \ref{main1}, then it has the WLP. 

Numerical computations have lead to the following conjecture: 

\begin{conjecture}
\label{conj}
Let $S$ be a numerical semigroup and let $A$ be the associated graded algebra of the Ap\'ery set of $S$. If $A$ is Complete Intersection or if the codimension of $A$ is 3, then any quotient of the form $ \dfrac{A}{(0:_A x_l)}$ has the WLP.
\end{conjecture}

Proving this conjecture to be true means to prove that, for these rings, Theorem \ref{main1} holds also in the case of even socle degree without further assumptions on the dimension of the graded components. We are going also to show that, in this particular case of rings associated to Ap\'ery sets, the colon ideal $(0:_G x_l)$ can be generated as ideal of $G$ by monomials. This does not happen in Example \ref{wachi}, because one of the minimal generators of $(0:_G z)$ in that case is the binomial $ay-bz.$

\bigskip

\section{Complete Intersections Algebras}

In this section we recall some results of D'Anna, Micale and Sammartano \cite{da-mi-sa} that we need to characterize when the graded algebra associated to the Ap\'ery Set of a numerical semigroup is a Complete Intersection.
Let $S= \langle g_1=m, g_2, \ldots, g_n \rangle $ be a numerical semigroup.
In \cite{da-mi-sa} the authors introduced two hyper-rectangles in $\N^{n-1}$ that contain the representations of the elements $\aap(S)$ and whose properties determine when the associated graded algebra $A=\bigoplus_{i = 0}^D A_i$ is a Complete Intersection.
At the end of this section, using a classical criterion for Weak Lefschetz properties of Complete Intersection algebras, we prove that any Complete Intersection algebra $A$ associated to the Ap\'ery Set of a numerical semigroup satisfies the Quotient Condition (given in Definition \ref{discesa}).

\begin{definition}
\label{gammabetadef}
For $2 \leq i \leq n$, define: 

$ \beta_i :=\max\{ h \in \mathbb{N} \, | \, hg_i \in \aap(S) \ \text{ and } \ \ord(h g_i)=h\};$ \\
$ \gamma_i :=\max\{ h \in \mathbb{N} \, | \, hg_i \in
\aap(S), \, \ \ord(h g_i)=h \  \text{ and }$  $ hg_i \text{  has a unique maximal
representation}\}.$
\end{definition}

The positive natural numbers $\beta_i$ and $ \gamma_i $ are strongly related to the degrees of the generators of the defining ideal of $A$ seen as quotient of the polynomial ring in $n-1$ variables.

\begin{remark}
\label{gamma2n}
For all $i=2, \ldots, n$, $\gamma_i \leq \beta_i$. But always $\gamma_2 = \beta_2$ and  $\gamma_n = \beta_n$.
\end{remark}

\begin{proof}
By definition $\gamma_i \leq \beta_i$ for every $i$. For the second statement assume $\gamma_2 < \beta_2$. Then we must have that $ (\gamma_2+1)g_2 = \sum_{j\neq 2}\lambda_jg_j $ are two different maximal representations of the same element of $\aap(S)$. Hence they have the same order and therefore $ (\gamma_2+1) = \sum_{j\neq 2}\lambda_j $, but this is impossible since $g_2 < g_3 < \ldots < g_n$. For the same reason it follows that $\gamma_n = \beta_n$.
\end{proof}

For the proofs of all the following facts see \cite[Section 2]{da-mi-sa}.

\begin{prop}
\label{raprmax}
Let $\omega= \sum_{i=2}^\nu \lambda_i g_i \in \aap(S)$ with $ \lambda= (\lambda_2 , \ldots, \lambda_n) $ a maximal representation.
Then $ \lambda_i \leq \beta_i$ for each $i$. If $ \lambda $ is the maximum of the set of maximal representations of $s$ with respect to the lexicographic order, then $ \lambda_i \leq \gamma_i$ for each $i$.
\end{prop}

\begin{definition}
\label{GammaBetadef}
Define two hyper-rectangles in $\N^{n-1}$: \\
$  B = \Big\{\sum_{i=2}^{n} \lambda_i g_i \, | \, 0 \leq \lambda_i
\leq \beta_i \Big\}$ and $ \Gamma = \Big\{\sum_{i=2}^{n} \lambda_i g_i \, | \, 0 \leq \lambda_i
\leq \gamma_i \Big\}$. 
\end{definition}

Using Proposition \ref{raprmax}, it can be proved that $$ \aap(S) \subseteq \Gamma \subseteq B $$ and moreover we can give some characterizations of when the possible equalities hold. 

\begin{prop}
\label{Beta}
The following assertions are equivalent for a numerical semigroup $S$: \begin{enumerate}
\item $\aap(S)= B$.
\item The maximal element of $\aap(S)$, $f+g_1$ has a unique maximal representation.
\item All the elements of $\aap(S)$ have a unique maximal representation.
\item $g_1= \prod_{i=2}^n (\beta_i+1)$.
\item $D=\ord(f+g_1)= \sum_{i=2}^n \beta_i$.
\end{enumerate}
\end{prop}

\begin{prop}
\label{Gamma}
The following assertions are equivalent for a numerical semigroup $S$: \begin{enumerate}
\item $\aap(S)= \Gamma$.
\item $g_1= \prod_{i=2}^n (\gamma_i+1)$.
\item $D=\ord(f+g_1)= \sum_{i=2}^n \gamma_i$.
\end{enumerate}
\end{prop}

Consider now the ring $A= \bigoplus_{i = 0}^D A_i \cong \dfrac{K[x_2,\ldots, x_n]}{I}$.  associated to $\aap(S)$. This ring is a Complete Intersection if and only if $I$ is minimally generated by $n-1$ elements. The next proposition is the key to characterizing when it happens.

\begin{prop}
\label{definingideal}
The defining ideal $I$ of $A$ always contains the ideal $$  \widetilde{\, I \,}= (x_i^{\gamma_i+1}-\rho_i\prod_{j\neq i}x_j^{\lambda_j}:\ i=2\dots, n )$$
where $ \rho_i=0$ if $ \beta_i=\gamma_i $ and $ \rho_i=1$ if $ \beta_i>\gamma_i $. In the second case
$(\gamma_i+1)g_i=\sum_{j\neq i}\lambda_jg_j$ are two different maximal representations of the same element of $\aap(S)$. \\
Furthermore $ I= \widetilde{\, I \,} $ if and only if $ A $ is a Complete Intersection.  
\end{prop}

This result is proved using the definitions of $\beta_i$ and $ \gamma_i $ and considering the monomials of $A$ corresponding to the representations of the elements of $S$. Let $s \in S$ and let $ \lambda=(\lambda_1, \ldots, \lambda_n) $ be a representation of $s$, the key fact is that the monomial $ x^{\lambda}:= x_1^{\lambda_1} \cdots x_n^{\lambda_n} \in I$ if either $s \not \in \aap(S)$ or if $\lambda$ is not maximal.

\begin{theorem}  
\label{mainCI}
The following assertions hold: 
\begin{enumerate}
\item $A$ is a Complete Intersection if and only if $ \aap(S) = \Gamma $.
\item $A$ is a Complete Intersection and its defining ideal $I$ is generated by monomials if and only if $ \aap(S) = B $.
\end{enumerate}
\end{theorem}

 \begin{proof}
 Observe that $ \widetilde{\, I \,} $ is an ideal of $ K[x_2,\ldots, x_n]$ generated by a regular sequence of $n-1$ elements, hence $ \dfrac{K[x_2,\ldots, x_n]}{\widetilde{\, I \,}}$ is an Artinian Complete Intersection ring and therefore it has finite dimension as $K$-vector space. \\
 Assertion 1 follows from the following inequality: $$ g_1= |\aap(S)|= \dim_K(\dfrac{K[x_2,\ldots, x_n]}{I}) \leq \dim_K(\dfrac{K[x_2,\ldots, x_n]}{\widetilde{\, I \,}}) = |\Gamma|= \prod_{i=2}^n (\gamma_i+1)$$ applying Proposition \ref{definingideal} and Proposition \ref{Gamma}. \\
 For assertion 2 just observe that if $ \aap(S)= B $, then $\gamma_i=\beta_i$ for all $i$ and use Proposition \ref{definingideal} and Proposition \ref{Beta}.
 \end{proof}

\begin{example}
\label{ex3}
Consider the numerical semigroup $ S= \langle 15, 21, 35 \rangle$. \\
Its Ap\'ery Set is $\aap(S)= \lbrace 0, 21, 35, 42, 56, 70, 63, 77, 91, 84, 98, 112, 119, 133, 154 \rbrace$. We can see that  $84=21 \cdot 4 \in \aap(S)$ and $ 105= 21 \cdot 5 \not \in \aap(S) $, then $70=35 \cdot 2 \in \aap(S)$ and $ 105= 35 \cdot 3 \not \in \aap(S) $, hence $\beta_2= \gamma_2=4$, $\beta_3= \gamma_3=2$ and we can verify that $$  B = \Big\{\sum_{i=2}^{3} \lambda_i g_i \, | \, 0 \leq \lambda_i
\leq \beta_i \Big\} = \aap(S). $$
The associated graded algebra is $$  A=\dfrac{K[y, z]}{(y^5, z^3)}. $$ and it is a monomial Complete Intersection.

\bigskip 

Consider now the semigroup $ S= \langle 8, 10, 11, 12 \rangle$ as in Example \ref{ex1}.

In this case $\aap(S)= \lbrace 0, 10, 11, 12, 21, 22, 23, 33 \rbrace$ and we can see that
 $20=10 \cdot 2 \not \in \aap(S)$, $ 24= 12 \cdot 2 \not \in \aap(S) $, $33=11 \cdot 3 \in \aap(S)$, $ 44= 11 \cdot 4 \not \in \aap(S) $ and $ 11 \cdot 2 = 10+12$. \\
 Hence $\beta_2= \gamma_2=1$, $\beta_4= \gamma_4=1$ but $1= \gamma_3 < \beta_3=3$. Therefore $$  \Gamma = \Big\{\sum_{i=2}^{4} \lambda_i g_i \, | \, 0 \leq \lambda_i
\leq \gamma_i \Big\} = \aap(S) \subsetneq B. $$
The associated graded algebra is $$ A=\dfrac{K[y, z, w]}{(y^2, z^2 - yw, w^2)}.$$ and it is a Complete Intersection but it is not monomial.
\end{example}

The next proposition is a standard result for WLP of Complete Intersection algebras that states that if there exists a minimal generator of the defining ideal having a sufficiently big degree (about half of the socle degree), then $A$ has the WLP. For the proof see \cite[3.52 and 3.54]{monografia}. 
  
 \begin{prop}
\label{feef}
 Let $A = \dfrac{K[x_1,\ldots, x_n]}{(f_1, \ldots, f_n)}$ a Complete Intersection standard graded Artinian $K$-algebra. \\ Call $d_i:=$ deg$f_i$ and assume that $d_n \geq d_i \geq 2$ for all $i$. Then, the condition $$ d_n \geq d_1 + \ldots + d_{n-1} - n $$ implies that $A$ has the WLP. 
\end{prop}

\begin{corollary}
\label{gammagrande}
Let $A= \bigoplus_{i = 0}^D A_i$ be a standard graded algebra associated to the Ap\'ery Set of a numerical semigroup. If $A$ is a Complete Intersection and there exist $\gamma_i \geq \dfrac{D-2}{2}$, then $A$ has the WLP.
\end{corollary}

\begin{proof}
Since $A$ is a Complete Intersection, $D=\sum_{i=2}^n \gamma_i$ by Proposition \ref{Gamma} and Theorem \ref{mainCI}.
Assume, by changing the order of the generators of the defining ideal of $A$, that $ \gamma_n \geq \gamma_i $ for all $ i $. Using notation of Proposition \ref{feef} we have $d_i= \gamma_i +1$. Thus $ W:= d_n - (d_2 + \ldots + d_{n-1} - n+1) = \gamma_n +1 - (\sum_{i=2}^n (\gamma_i+1)- \gamma_n - 1-n+1) =  \gamma_n +1 -(D+n-1 -\gamma_n -n)= 2\gamma_n -D +2 $. By assumption $W \geq 0$ and hence by Proposition \ref{feef}, $A$ has the WLP.
\end{proof}

\begin{theorem}
\label{teo2}
Let $A= \bigoplus_{i = 0}^D A_i$ be the graded algebra associated to the Ap\'ery Set of a numerical semigroup. If $A$ is a Complete Intersection, then $A$ satisfies the Quotient Condition. Moreover the corresponding colon ideal defining $A$ as a quotient of an algebra with the WLP is generated by monomials.
\end{theorem}

\begin{proof}
By Proposition \ref{definingideal}, the defining ideal of $A$ is $$(f_2, \ldots, f_n):=(x_i^{\gamma_i+1}-\rho_i\prod_{j\neq i}x_j^{\lambda_j}:\ i=2\dots, n )$$ and we recall that $\rho_2=\rho_n=0$ by Remark \ref{gamma2n}. \\
By Corollary \ref{gammagrande}, if there exists $\gamma_i \geq t= \frac{D-2}{2}$, then $A$ has the WLP and hence it satisfies the Quotient Condition with $C=0$. Hence assume $\gamma_i < t $ for every $i$ and consider the Artinian Complete Intersection ring $$  B:= \dfrac{K[x_2,\ldots, x_n]}{(x_2^N, f_3, \ldots, f_n)}$$ with $N \geq D- \gamma_2$. By Proposition \ref{Gamma}, $D=\sum_{i=2}^n \gamma_i$ and therefore the socle degree of $B$ is by construction $E=D- \gamma_2+N-1$. Now $$ \dfrac{E}{2} = \dfrac{D-\gamma_2+N-1}{2} \leq \dfrac{2N-1}{2} < N$$ and thus, Corollary \ref{gammagrande} implies that $B$ has the WLP. \\
Call $ T:= \dfrac{K[x_2,\ldots, x_n]}{(f_3, \ldots, f_n)} $. Since for every $i \neq 2$, $\deg(x_2, f_i) \leq \gamma_2$, we have $ A \cong \dfrac{T}{(x_2^{\gamma_2+1})} $ and $ B \cong \dfrac{T}{(x_2^{N})} $. Hence $$ A \cong \dfrac{B}{(0:_B x_2^{N-\gamma_2-1})} $$ satisfies the Quotient Condition. 

What written above implies that the ideal $ (0:_B x_2^{N-\gamma_2-1}) $ is the principal ideal $x_2^{\gamma_2 +1}B$.
\end{proof}


\bigskip

\section{Codimension 3 Algebras}

In this section we study the Gorenstein graded algebras associated to the Ap\'ery Set of numerical semigroups in low codimension. 

Let $S $ be an $M$-pure symmetric numerical semigroup and let $A$ be the graded algebra associated to $\aap(S)$. Observe that if $S$ is generated by $n$ elements, then the codimension of the ring $A$ is $n-1$. In \cite{da-mi-sa} is proved that, when $S$ is generated by $3$ elements, it is $M$-pure symmetric if and only if $\aap(S)$ is equal to the hyper-rectangle $B$ and hence $A$ is Gorenstein if and only if it is a monomial Complete Intersection. In this case it is known that $A$ has the SLP.

A more interesting case that we are going to discuss is when $S$ is generated by $4$ natural numbers $g_1, g_2, g_3, g_4$. Write $$ A \cong \dfrac{K[y, z, w]}{I}. $$ In this context $A$ has codimension $3$ and, if it is a Complete Intersection, it has the WLP. But we recall that in general it is not known if a Gorenstein Artinian algebra of codimension $3$ has the WLP. 

Therefore for the rest of the section we assume that $A$ is not a Complete Intersection. This means by Theorem \ref{mainCI} that $\aap(S)$ is properly contained in the hyper-rectangle $\Gamma$. We want to characterize the defining ideal of $A$. In order to do this, we need some more results.

\begin{lemma}
\label{gammaminore}
Let $S = \langle g_1, g_2, g_3, g_4 \rangle$ be an $M$-pure symmetric numerical semigroup and assume $ \aap(S) \subsetneq \Gamma$. Then $\gamma_3 < \beta_3$.
\end{lemma}

\begin{proof}
Since $\aap(S) \subsetneq \Gamma \subseteq B$, by Proposition \ref{Beta}, the maximal element of $\aap(S)$, $f+g_1$ has more than one maximal representation. Subtracting common terms for two of these representations we obtain a double representation of an element of $\aap(S)$ that has to be necessarily of the form $ \lambda_3 g_3= \mu_2 g_2 + \mu_4 g_4 $ (since $g_2 < g_3 < g_4$) and the two different representations have the same order. This implies $\gamma_3 < \beta_3$ by definition.
\end{proof}

\begin{remark}
\label{codim4}
The fact that $A$ is not a Complete Intersection implies that there exists $\gamma_i < \beta_i$ for some $i=2, \ldots, n$, is not true in codimension higher than $3$. This is due to the fact that there may appear double maximal representations of elements of $\aap(S)$ like $\sum_{i \in \mathcal{I}} \lambda_i g_i = \sum_{j \in \mathcal{J}} \mu_j g_j$ with $ \mathcal{I}, \mathcal{J} \subseteq \{ 2, \ldots, n \} $, $  \mathcal{I} \cap \mathcal{J}= \emptyset $ and $  |\mathcal{I}|, |\mathcal{J}| \geq 2 $. For example consider the numerical semigroup $ S= \langle g_1, g_2, g_3, g_4, g_5 \rangle = \langle 6, 7, 8, 9, 10 \rangle$. The Ap\'ery Set of $S$ is $ \aap(S)= \lbrace 0, 7, 8, 9, 10, 17 \rbrace $ and $S$ is $M$-pure symmetric. Observe that $2g_i \not \in \aap(S)$ for every $i=2,3,4,5$ and hence $\gamma_i=\beta_i=1$. Moreover $15 = 7+8= 6+9 \in \Gamma \setminus \aap(S)$. So in this case we have $\aap(S) \subsetneq \Gamma = B$ and $A$ is not a Complete Intersection.
\end{remark}

\begin{corollary}
\label{tilde}
There exists one element $s$ of the Ap\'ery Set of the numerical semigroup $S$ which has the double representation $$ s= (\gamma_3+1) g_3= \mu_2 g_2 + \mu_4 g_4 $$ with $ 1 \leq \mu_2 \leq \gamma_2 $, $  1 \leq \mu_4 \leq \gamma_4 $ and $\mu_2+\mu_4= \gamma_3 +1$. Hence the ideal $$\widetilde{\, I \,}=(y^{\gamma_2+1}, z^{\gamma_3+1}-y^{\mu_2}w^{\mu_4},  w^{\gamma_4+1})$$ is properly contained in $I$.
\end{corollary}

\begin{proof}
This follows immediately from the previous Lemma \ref{gammaminore} and from Proposition \ref{definingideal} and Remark(\ref{gamma2n}). Notice that by Definition \ref{gammabetadef}, this element $ (\gamma_3+1) g_3 $ is the minimal in $\aap(S)$ with a double representation. The containment $ \widetilde{\, I \,} \subsetneq I $ is proper since $A$ is not a Complete Intersection.
\end{proof}

\begin{definition}
\label{GG}
Define the ring $G:=\dfrac{K[y, z, w]}{\widetilde{\, I \,}}.$ \\
We observe that $G$ is a standard graded Artinian Complete Intersection algebra and $$G=\langle x^{\lambda} \ | \sum_{i = 2}^4 \lambda_i g_i \in \Gamma \mbox{ and } \lambda \mbox{ is maximal }\rangle_K.$$ Therefore $A$ is isomorphic to a $K$-vector subspace of $G$. \\
As rings $A$ is a quotient of $G$ and we write $$ A \cong \dfrac{G}{J}$$ where $J:= \widetilde{\, I \,}/I.$
\end{definition}

We need to find which are the generators of the ideal $ J $ and for this reason, we want to characterize the elements of the set $ \Gamma \setminus \aap(S) $. Write $G=\bigoplus_{i = 0}^D G_i$ and $ A=\bigoplus_{i = 0}^{D-C} A_i $ for a positive integer $C$. This two rings are equal if and only if $A$ is a Complete Intersection and if and only if $C=0$. The first "if and only if" follows from Theorem \ref{mainCI}. We are going to explain the second one in the next proposition.

Call $\omega_D= \gamma_2 g_2+\gamma_3 g_3+\gamma_4 g_4$ the maximal element of $\Gamma$ and $\omega_E= f+g_1$ the maximal element of $\aap(S)$.

\begin{prop}
\label{GmenoA}
The following statements hold: 
\begin{enumerate}
\item $ \omega_E = \omega_D - Cg_3 $. 
 \item $ C \leq \gamma_3. $ 
 \item $ \Gamma \setminus \aap(S) = \{ \omega \in \Gamma \mbox{ : } \omega + Cg_3 \not \in \Gamma \}$.
\end{enumerate}
 \end{prop}

\begin{proof}
1. We set, for each $ \omega \in \Gamma $, the element $ \omega^{\prime}:= \omega_E - \omega $. Clearly when $\omega \in \aap(S)$, $ \omega^{\prime} \in \aap(S) $ and $\ord(\omega)+\ord(\omega^{\prime})= \ord(\omega_E)$ because $S$ is $M$-pure symmetric. Instead, when $\omega \not \in \aap(S)$, also $\omega^{\prime} \not \in \aap(S)$. Also the set $\Gamma$ is clearly symmetric by construction: $\omega_D - \upsilon \in \Gamma$ for each $ \upsilon \in \Gamma $ and also in this case an analogous equality for the orders of the elements holds, meaning that $\ord(\upsilon)+\ord(\omega_D - \upsilon)= \ord(\omega_D)$. 

By way of contradiction assume $\omega_E= \omega_D - g_2 - \upsilon$ with $ \upsilon \in \Gamma $. Thus $ \omega_E + g_2= \omega_D - \upsilon \in \Gamma$. But, by definition $\gamma_2 g_2 \in \aap(S)$ and, since by Remark \ref{gamma2n}, $\gamma_2= \beta_2$, it follows that $\omega_E= \gamma_2 g_2 + \lambda_3 g_3 + \lambda_4 g_4$. This fact implies $ \omega_E + g_2 \not \in \Gamma $. \\ 
In the same way, using that $ \gamma_4 g_4 \in \aap(S) $ and $\gamma_4=\beta_4$, it is possible to show that $ \omega_E $ cannot be of the form $\omega_D - g_4 - \upsilon$ with $ \upsilon \in \Gamma $ and hence $ \omega_E = \omega_D - Cg_3 $ with $C= \ord(\omega_D)-\ord(\omega_E)$. \\
2. Follows immediately from item 1 by the definition of $\omega_D$. \\
3. Set $W:=\{ \omega \in \Gamma \mbox{ : } \omega + Cg_3 \not \in \Gamma \}$ and take $\omega \in W$. 
By item 1, we can write for each $ \omega \in \Gamma $, $ \omega^{\prime}= \omega_D - Cg_3 - \omega $.
Assuming $ \omega \in \aap(S) $, we also have $ \omega^{\prime} \in Ap(S) \subseteq \Gamma $. Therefore by definition, $ \omega + Cg_3 = \omega_D - \omega^{\prime} \in \Gamma$ and this is a contradiction since $ \omega \in W $. This proves $ W \subseteq \Gamma \setminus \aap(S)$, now we prove the reverse inclusion. 

Take now $ \omega \in \Gamma \setminus \aap(S) $. As said before, we have in this case $ \omega^{\prime}= \omega_E - \omega= \omega_D - Cg_3- \omega \not \in \aap(S)$. If we assume $ \omega \not \in W$, we have $ \omega + Cg_3 \in \Gamma $ and hence $ \omega^{\prime} \in \Gamma$. \\
Thus $ \omega^{\prime} \in \Gamma \setminus \aap(S) $ and, by definition of $\aap(S)$, for every $\overline{\omega} \in \Gamma$, $ \overline{\omega}+ \omega^{\prime} \not \in \aap(S) $. But certainly there must exist a minimal generator of the semigroup $g_j \in \aap(S)$ (with $j \neq 1$) such that $ \omega - g_j \in \Gamma $ and, setting $\overline{\omega}:= \omega - g_j$, we have $ \overline{\omega}+ \omega^{\prime}= \omega_D - Cg_3 - g_j = \omega_E - g_j \in \aap(S)$. Therefore we must have $ \omega \in W $.
\end{proof}

\begin{theorem}
\label{colon}
Assume with the same notations as before, $G=\bigoplus_{i = 0}^D G_i$ and $ A=\bigoplus_{i = 0}^{D-C} A_i $. Set $ h_2= \gamma_2- \mu_2 +1  $, $ h_3= \gamma_3 - C +1  $ and $ h_4= \gamma_4- \mu_4 +1  $ where $ \mu_2 $ and $ \mu_4 $ are given in Corollary \ref{tilde}. Thus the defining ideal of $A$ is $$ I= \widetilde{\, I \,} +  (z^{h_3}y^{h_2}, z^{h_3}w^{h_4}).$$
Moreover $$ A = \dfrac{G}{(0:_G z^C)} $$ and therefore it satisfies the Quotient Condition (Definition \ref{discesa}).
\end{theorem}

\begin{proof}
We have seen that $A$ is isomorphic to $G$ modulo the ideal $J:= \widetilde{\, I \,}/I.$ By construction of these two algebras, the ideal $J$ is generated by the elements $ \{ x^{\lambda} \ | \sum_{i = 2}^4 \lambda_i g_i \in \Gamma \setminus \aap(S) \mbox{ and } \lambda \mbox{ is maximal } \} $. \\
Hence we need to show that $ y^{h_2} z^{h_3}$ and $ z^{h_3} w^{h_4} $ are the unique monomial representations of the minimal elements of $ \Gamma \setminus \aap(S) $ with respect to the standard partial order of $\Gamma \subseteq \N^3$. \\
Take $ \omega= \lambda_2 g_2 + \lambda_3 g_3 + \lambda_4 g_4 \in \Gamma \setminus \aap(S) $, thus by Proposition \ref{GmenoA} $ \omega + Cg_3 \not \in \Gamma $ and hence $ \lambda_3 + C \geq \gamma_3 +1 $. Moreover, to be outside of $\Gamma$ we need either $ \lambda_2 \geq \gamma_2 - \mu_2 +1$ or $ \lambda_4 \geq \gamma_4 - \mu_4 +1$. Indeed $ Cg_3 +  h_2 g_2 + h_3 g_3= Cg_3 + (\gamma_2- \mu_2 +1)g_2 + (\gamma_3-C+1) g_3= (\gamma_3+1) g_3 + (\gamma_2- \mu_2 +1)g_2= \mu_2 g_2 + \mu_4 g_4 + (\gamma_2- \mu_2 +1)g_2= \mu_4 g_4 + (\gamma_2+1)g_2 \not \in \Gamma$. Similarly we obtain $ Cg_3 +  h_4 g_4 + h_3 g_3 = \mu_2 g_2 + (\gamma_4+1)g_4 \not \in \Gamma$. \\
Therefore the minimal elements of $ \Gamma \setminus \aap(S) $ are $ h_2 g_2 + (\gamma_3-C+1) g_3 $ and $ h_4 g_4 + (\gamma_3-C+1) g_3 $ and this completes the proof of the first part of the Theorem. \\
For the "moreover" statement notice that, taking $f \in G$ homogeneous, then $f \in (0:_G z^C)$ if and only if $z^Cf \in \widetilde{\, I \,}$ and this happens if and only if the monomials of $f$ are corresponding to element of $\Gamma \setminus \aap(S)$, i.e., $f \in (z^{h_3}y^{h_2}, z^{h_3}w^{h_4})$.
\end{proof}

Theorem \ref{colon} also shows that the ideal $ (0:_G z^C) $ defining $A$ as a quotient of $G$ is a monomial ideal generated by $z^{h_3}y^{h_2}$ and $ z^{h_3}w^{h_4}$.



\bigskip

\section*{Acknowledgements}
The author wishes to thank Professor Akihito Wachi for providing him Example \ref{wachi}, showing that the passage of the WLP to a quotient by a colon ideal may in general fail. The author also wishes to thank his PhD advisor Professor, Marco D'Anna for his guidance during the preparation of this paper. 


\end{document}